\documentclass[11pt,reqno,fleqn]{article}
\usepackage{amsfonts,amssymb,amsmath,amsthm}
\usepackage[top=20mm,bottom=20mm,left=20mm,right=20mm]{geometry}
\usepackage[T1]{fontenc}
\usepackage{times}
\usepackage{multicol}
\usepackage{amsfonts}
\usepackage{graphicx}
\usepackage{graphics}
\usepackage{enumitem}
\usepackage{relsize}
\usepackage{cite}
\usepackage{fancyhdr}
\usepackage{color}
\usepackage{cases}
\fancypagestyle{titlepage}
{
   \fancyhf{}
   \lhead{The final publication is available at IOS Press through http://dx.doi.org/10.3233/JIFS-190632}
   \rhead{}
   \fancyfoot[C]{\thepage}
}

\newcommand\blfootnote[1]{%
  \begingroup
  \renewcommand\thefootnote{}\footnote{#1}%
  \addtocounter{footnote}{-1}%
  \endgroup
}
\newtheorem{theorem}{Theorem}[section]
\numberwithin{equation}{section}
\newtheorem{definition}[theorem]{Definition}
\usepackage{etoolbox}
\makeatletter
\patchcmd{\@maketitle}{\begin{center}}{\begin{flushleft}}{}{}
\patchcmd{\@maketitle}{\begin{tabular}[t]{c}}{\begin{tabular}[t]{@{}l}}{}{}
\patchcmd{\@maketitle}{\end{center}}{\end{flushleft}}{}{}
\makeatother
\begin{document}

\title{On the resummation of series of fuzzy numbers via generalized Dirichlet and generalized factorial series}
\author{Enes Yavuz}
\date{{\small Department of Mathematics, Manisa Celal Bayar University, Manisa, Turkey.\\ E-mail: enes.yavuz@cbu.edu.tr}}
\maketitle
\thispagestyle{titlepage}
\blfootnote{\rule{0.63cm}{0cm}\emph{\!\!Mathematics Subject Classification:} 03E72, 40A05, 40E05}

\noindent\textbf{Abstract:}
\noindent We introduce semicontinuous summation methods for series of fuzzy numbers and give Tauberian conditions under which summation of a series of fuzzy numbers via generalized Dirichlet series and via generalized factorial series implies its convergence. Besides, we define the concept of level Fourier series of fuzzy valued functions and obtain results concerning the summation of level Fourier series.

\section{Introduction}\allowdisplaybreaks
\noindent Following their invention, Dirichlet and factorial series have fascinated many researchers  and found applications in various fields of mathematics. Introduced primarily in real variables to answer combinatorial questions, Dirichlet series have played important roles in analytic number theory. The Riemann zeta function $\zeta(s)$ is represented by Euler product over primes and has deep relation with prime number theorem which is, in fact, equivalent to the fact that there are no zeros of the zeta function on the line $Re(s)=1$. Dirichlet series and in particular the Riemann zeta function serve also as a tool for number theorists to study behaviors of arithmetic functions and to derive related identities\cite{arithmetic1,arithmetic2}. Beside the usage in analytic number theory, they arise as expansions of functions and as solutions of differential equations recently\cite{dsolut1, dsolut2,dsolut3,dsolut4}. On the other hand, factorial series are studied with regard to their analytical properties and used in transformation of series and in expansions of functions. Authors converted power series to factorial series of more rapid convergence and given factorial series expansions for some functions\cite{factor1,factor2,factor3,factor4}. There are also studies concerning factorial series solutions of difference equations\cite{factdiff1,factdiff2,factdiff3}. In connection with the context of the paper, both series are also utilized in regularization of divergent series\cite{factor2,factor4,factdiff2,sum1,sum2}.

As extension of classical sets, fuzzy sets are introduced by Zadeh\cite{zadeh} to represent imprecise knowledge which is hard to handle by using classical sets. Since its introduction, researchers used fuzzy sets as a smart mathematical tool for modeling real-world problems and achieved satisfactory results. In the meantime, theoretical basis of the concept have also developed and many concepts in classical mathematics are extended to fuzzy mathematics. In particular, sequences and series of fuzzy numbers are introduced and corresponding convergence properties are investigated\cite{seq1,seq2,seq3}. Furthermore, authors introduced summability methods to handle series of fuzzy numbers which fail to converge in fuzzy number space and given various Tauberian conditions to achieve the convergence\cite{ek1,ek2,ek3,abel,ek4,ekyavuz,ekdutta,eket}. Following the achievements in the literature we now introduce a general summation method for series of fuzzy numbers and obtain Tauberian conditions under which summation of a series of fuzzy numbers via generalized Dirichlet and generalized factorial series imply its ordinary sum. Besides, we define the concept of level Fourier series of fuzzy valued functions and prove that level Fourier series of fuzzy valued functions are summable via generalized Dirichlet series to the original function. Results for some particular cases of the general summation method are also deduced as corollaries.
\section{Preliminaries}
\noindent A \textit{fuzzy number} is a fuzzy set on the real axis, i.e. u is normal, fuzzy convex, upper semi-continuous and $\operatorname{supp}u =\overline{\{t\in\mathbb{R}:u(t)>0\}}$ is compact \cite{zadeh}.
$\mathbb{R}_{\mathcal{F}}$ denotes the space of fuzzy numbers. \textit{$\alpha$-level set} $[u]_\alpha$
is defined by
\begin{eqnarray*}
[u]_\alpha:=\left\{\begin{array}{ccc}
\{t\in\mathbb{R}:u(t)\geq\alpha\}&& \ \textrm{if}\quad 0<\alpha\leq 1 \\[6 pt]
\overline{\{t\in\mathbb{R}:u(t)>\alpha\}}&&\!\!\!\!\!\!\textrm{if}\quad \alpha=0.\end{array} \right.
\end{eqnarray*}
$r\in\mathbb{R}$ may be seen as a fuzzy number $\overline{r}$ defined by
\begin{eqnarray*}
\overline{r}(t):=\left\{\begin{array}{ccc}
1 && \ \textrm{if} \quad t=r \\
0 && \ \ \textrm{if} \quad t\neq r.\end{array} \right.
\end{eqnarray*}
\begin{theorem}[Representation Theorem]\label{rep}\cite{gv}
Let $[u]_\alpha\!\!=\![u^-(\alpha),u^+(\alpha)]$ for $u\in
\mathbb{R}_{\mathcal{F}}$ and for each $\alpha\in[0,1]$. Then, the following statements
hold:

(i) $u^-(\alpha)$ is a bounded and non-decreasing left continuous
function on $(0,1]$.

(ii) $u^+(\alpha)$ is a bounded and non-increasing left continuous
function on $(0,1]$.

(iii) The functions $u^-(\alpha)$ and $u^+(\alpha)$ are right
continuous at the point $\alpha=0$.

(iv) $u^-(1)\leq u^+(1)$.

\noindent Conversely, if the pair of
functions $\gamma$ and $\beta$ satisfies the conditions (i)-(iv),
then there exists a unique $u\in \mathbb{R}_{\mathcal{F}}$ such that
$[u]_\alpha:=[\gamma(\alpha),\beta(\alpha)]$ for each
$\alpha\in[0,1]$. The fuzzy number $u$ corresponding to the pair of
functions $\gamma$ and $\beta$ is defined by $u:\mathbb{R}\to[0,1]$,
$u(t):=\sup\{\alpha:\gamma(\alpha)\leq t\leq \beta(\alpha)\}$.
\end{theorem}
Let $u,v\in \mathbb{R}_{\mathcal{F}}$ and $k\in\mathbb{R}$. The addition and scalar multiplication are defined by
\begin{eqnarray}\label{operations}
(i) \ [u+v]_\alpha=[u^-_{\alpha}+v^-_{\alpha}, u^+_{\alpha}+v^+_{\alpha}] \qquad(ii)\  [k u]_\alpha=\begin{cases}[ku^-_{\alpha}, ku^+_{\alpha}] \ \  \textrm{if}\ \  k\geq 0\\ [ku^+_{\alpha},ku^-_{\alpha}] \ \ \textrm{if}\ \ k< 0\end{cases}
\end{eqnarray}
where $[u]_\alpha=[u^-_{\alpha}, u^+_{\alpha}]$, for all $\alpha\in[0,1]$.

Fuzzy number $\overline{0}$ is identity element in $(\mathbb{R}_{\mathcal{F}}, +)$ and none of $u \neq \overline{r}$ has inverse in $(\mathbb{R}_{\mathcal{F}}, +)$. For any $k_1,k_2 \in\mathbb{R}$ with $k_1k_2\geq 0$, distribution property $(k_1 + k_2) u = k_1u + k_2u$ holds but for general $k_1,k_2 \in\mathbb{R}$ it fails to hold. On the other hand properties $k (u+ v) = ku + kv$ and $k_1 (k_2 u) = (k_1 k_2) u$ holds for any $k, k_1,k_2 \in\mathbb{R}$. It should be noted that $\mathbb{R}_{\mathcal{F}}$ with addition and scalar multiplication defined above is not a linear space over $\mathbb{R}$.

Partial ordering relation on $\mathbb{R}_{\mathcal{F}}$ is defined as follows:
\begin{eqnarray*}
u\preceq v \Longleftrightarrow
[u]_{\alpha}\preceq[v]_{\alpha}\Longleftrightarrow u^-_{\alpha}\leq
v^-_{\alpha}~\text{ and }~u^+_{\alpha}\leq v^+_{\alpha}~\text{ for
all }~\alpha\in[0,1].
\end{eqnarray*}

The metric $D$ on $\mathbb{R}_{\mathcal{F}}$ is defined as
\begin{eqnarray*}
 D(u,v):=
\sup_{\alpha\in[0,1]}\max\{|u^-_{\alpha}-v^-_{\alpha}|,|u^+_{\alpha}-
v^+_{\alpha}|\},
\end{eqnarray*}
and it has the following properties
\begin{eqnarray*}
D(ku,kv)=|k|D(u,v), \quad\quad D(u+v,w+z)\leq D(u,w)+D(v,z)
\end{eqnarray*}
where $u,v,w,z\in \mathbb{R}_{\mathcal{F}}$ and $k\in\mathbb{R}$.
\begin{definition}\label{series}\cite{kim}
Let $(u_k)$ be a sequence of fuzzy numbers. Denote $s_n=\sum_{k=0}^nu_k$ for
all $n\in\mathbb{N}$, if the sequence $(s_n)$ converges to a fuzzy
number $u$ then we say that the series $\sum u_k$ of fuzzy numbers
converges to $u$ and write $\sum u_k=u$ which implies that
\begin{eqnarray*}
\sum_{k=0}^{n}u^{-}_{k}(\alpha)\rightarrow u^{-}(\alpha) ~\text{
and }~ \sum_{k=0}^{n}u^{+}_{k}(\alpha)\rightarrow u^{+}(\alpha)\qquad\qquad (n\to\infty)
\end{eqnarray*}
uniformly in $\alpha\in [0,1]$. Conversely, for the sequence
$(u_{k})$ of fuzzy numbers if  \ $\sum_{k}u^{-}_{k}(\alpha)=\gamma(\alpha)$ and
$\sum_{k}u^{+}_{k}(\alpha)=\beta(\alpha)$ converge uniformly in
$\alpha$, then $\{(\gamma(\alpha),\beta(\alpha)): \alpha\in
[0,1]\}$ defines a fuzzy number $u$ represented by $[u]_{\alpha}=[\gamma(\alpha),\beta(\alpha)]$ and $\sum u_{k}=u$.
\end{definition}
\noindent
Besides, we say that series $\sum u_k$ is bounded if the sequence $(s_n)$ is bounded. The set of bounded series of fuzzy numbers is denoted by  $bs(F)$.
\begin{theorem}\label{triangle}\cite{tb}
If $\sum u_k$ and $\sum v_k$ converge, then $D\left(\sum u_k, \sum v_k\right)\leq \sum D(u_k, v_k).$
\end{theorem}
\noindent A fuzzy valued function $f:I \subset\mathbb{R}\to \mathbb{R}_{\mathcal{F}}$ has the parametric representation
\begin{eqnarray*}
\left[f(x)\right]_{\alpha}=\left[f^-_{\alpha}(x),f^+_{\alpha}(x)\right],
\end{eqnarray*}
for each $x\in I$ and $\alpha\in[0,1]$.

A function $f:\mathbb{R}\to \mathbb{R}_{\mathcal{F}}$ is $2\pi$-periodic if $f(x)=f(x+2\pi)$ for all $x\in\mathbb{R}$. The space of all $2\pi$-periodic and continuous fuzzy valued functions on $\mathbb{R}$ is denoted by $C_{2\pi}^{(\mathcal{F})}(\mathbb{R})$ where the continuity is meant with respect to metric $D$. Besides the space of all $2\pi$-periodic and real valued continuous functions on $\mathbb{R}$ is denoted by $C_{2\pi}(\mathbb{R})$ and equipped with the supremum norm $\Vert\cdot\Vert$. Here we note that if $f\in C_{2\pi}^{(\mathcal{F})}(\mathbb{R})$ then $f^{\mp}_{\alpha}\in C_{2\pi}(\mathbb{R})$ in view of metric $D$.

Let $f,g: I \to \mathbb{R}_{\mathcal{F}}$ be fuzzy valued functions. The distance between $f$ and $g$ is defined by
\begin{eqnarray*}
 D^*(f,g)=\sup_{x\in I} D(f(x),g(x))=\sup_{x\in I}\sup_{\alpha\in[0,1]}\max\{|f^-_{\alpha}(x)-g^-_{\alpha}(x)|,|f^+_{\alpha}(x)-
g^+_{\alpha}(x)|\}.
\end{eqnarray*}
Let $L:C_{\mathcal{F}}(\mathbb{R})\to C_{\mathcal{F}}(\mathbb{R})$ be an operator where $C_{\mathcal{F}}(\mathbb{R})$ denotes the space of all continuous fuzzy valued functions on $\mathbb{R}$. Then, we call $L$ a {\it fuzzy linear operator} iff
\begin{eqnarray*}
L(c_1 f_1+ c_2 f_2)=c_1 L(f_1)+ c_2 L(f_2),
\end{eqnarray*}
for any $c_1,c_2\in \mathbb{R}, f_1,f_2\in C_{\mathcal{F}}(\mathbb{R})$. Also, operator $L$ is called {\it fuzzy positive linear operator} if it is fuzzy linear and the condition $L(f;x)\preceq L(g;x)$ is satisfied for any $f,g\in C_{\mathcal{F}}(\mathbb{R})$ with $f(x)\preceq g(x)$ and for all $x\in\mathbb{R}$ .
\begin{theorem}\cite{anastassiou}\label{theorem1}
Let $\left\{L_n\right\}_{n\in\mathbb{N}}$ be a sequence of fuzzy positive linear operators from $C_{2\pi}^{(\mathcal{F})}(\mathbb{R})$ into  itself. Assume that there exists a corresponding sequence $\left\{\widetilde{L}_n\right\}_{n\in\mathbb{N}}$ of positive linear operators defined on $C_{2\pi}(\mathbb{R})$ with the property
\begin{eqnarray}\label{property1}
\left\{L_n(f;x)\right\}_{\alpha}^{\mp}=\widetilde{L}_n\left(f_{\alpha}^{\mp};x\right),
\end{eqnarray}
for all $x\in\mathbb{R},\alpha\in[0,1],n\in\mathbb{N}$ and $f\in C_{2\pi}^{(\mathcal{F})}(\mathbb{R})$. Assume further that
\begin{eqnarray}\label{lazým}
\lim_{n\to\infty}\Vert\widetilde{L}_n(f_i)-f_i\Vert=0,
\end{eqnarray}
for $i=0,1,2$ with $f_0(x)=1, f_1(x)=\cos x, f_2(x)=\sin x$. Then, for all $f\in C_{2\pi}^{(\mathcal{F})}(\mathbb{R})$ we have
\begin{eqnarray*}
\lim_{n\to\infty}D^{*}(L_n(f),f)=0.
\end{eqnarray*}
\end{theorem}

\section{Main results}
\noindent We now introduce semicontinuous summation methods of series of fuzzy numbers, defined by means of nonnegative and uniformly bounded sequences of continuous real valued functions on $[0,\infty)$. Introduced method acts on the terms of the series of fuzzy numbers directly and does not require computation of the partial sums which may be challenging even in the case of series of real numbers. The method also includes many known summation methods.
\begin{definition}
Let $\sum u_n$ be a series of fuzzy numbers and $\left\{\phi_n(s)\right\}$ be a nonnegative uniformly bounded sequence of continuous real valued functions defined on $[0,\infty)$ such that $\phi_{n+1}(s)\leq \phi_n(s)$ for all $s\in[0,\infty)$ and $\phi_n(0)=1$. Then, series $\sum u_n$ of fuzzy numbers is said to be $(\phi)$ summable to fuzzy number $\mu$ if $\sum u_n\phi_n(s)$ exists for $s>0$ and
\begin{eqnarray*}
\lim_{s\to 0^+}\sum_{n=0}^{\infty}u_n\phi_n(s)=\mu.
\end{eqnarray*}
\end{definition}
\begin{theorem}\label{regularity}
If series $\sum u_n$ of fuzzy numbers converges to fuzzy number $\mu$, then it is $(\phi)$ summable to $\mu$.
\end{theorem}
\begin{proof}
Let $\sum u_n$ be a series of fuzzy numbers and $\sum u_n=\mu$. It is sufficient to show that series $\sum u_n\phi_n(s)$ exists for $s>0$ and
\begin{eqnarray*}
\lim_{s\to 0^+}\sum_{n=0}^{\infty}u_n\phi_n(s)=\mu.
\end{eqnarray*}
From Definition \ref{series}, series $\sum u_n\phi_n(s)$ exists for $s>0$ if series $\sum u_n^{\mp}(\alpha)\phi_n(s)$ converge uniformly in $\alpha\in[0,1]$. By convergence of the series $\sum u_n$ we know that $\sum u_n^{\mp}(\alpha)$ converge uniformly in $\alpha\in[0,1]$. Besides, $\left\{\phi_n(s)\right\}$ is uniformly bounded and $\phi_{n+1}(s)\leq \phi_n(s)$. So by Abel's uniform convergence test, series $\sum u_n^{\mp}(\alpha)\phi_n(s)$ converge uniformly in $\alpha$. Then, series $\sum u_n\phi_n(s)$ converges in fuzzy number space for $s>0$. Now we aim to show that
\begin{eqnarray*}
\lim_{s\to 0^+}\sum_{n=0}^{\infty}u_n\phi_n(s)=\mu.
\end{eqnarray*}
Define the sequence of continuous fuzzy valued functions $\left\{f_n(s)\right\}$ on $[0,1]$ such that $f_n(s)=\sum_{k=0}^nu_k\phi_k(s)$. Then, $\lim_{n\to\infty}f_n(s)=\sum_{k=0}^{\infty} u_k\phi_k(s)=f(s)$. From Theorem 3.3 and Theorem 3.5 in \cite{duzgunluk}, $f$ is continuous on $[0,1]$ if and only if sequence $\left(\left[f_n(s)\right]^{\mp}(\alpha)\right)$ uniformly converge to $\left(\left[f(s)\right]^{\mp}(\alpha)\right)$ on $[0,1]\times[0,1]$. Series $\sum u_k^{\mp}(\alpha)\phi_k(s)$ converge uniformly on $[0,1]\times[0,1]$ by Abel's uniform convergence test in view of the facts that
\begin{itemize}
\item $\sum u_k^{\mp}(\alpha)$ converge uniformly on $[0,1]\times[0,1]$,
\item $\left\{\phi_k(s)\right\}$ is uniformly bounded  on $[0,1]\times[0,1]$ and $\phi_{k+1}(s)\leq \phi_k(s)$ for $(\alpha,s)\in [0,1]\times[0,1]$.
\end{itemize}
So, $f(s)=\sum u_k\phi_k(s)$ is continuous on $[0,1]$ which implies
\begin{eqnarray*}
\lim_{s\to 0^+}\sum_{k=0}^{\infty}u_k\phi_k(s)=\sum_{k=0}^{\infty}u_k=\mu.
\end{eqnarray*}
This completes the proof.
\end{proof}
In view of the result above we see that convergence of a series of fuzzy numbers implies its $(\phi)$ summability in fuzzy number space. However $(\phi)$ summability of a series of fuzzy numbers may not imply its convergence. Motivated by this fact now we aim to investigate conditions under which $(\phi)$ summability of a series of fuzzy numbers imply convergence in fuzzy number space in particular cases $\phi_n(s)=e^{-\lambda_ns}$ and $\phi_n(s)=\frac{\lambda_0\lambda_1\ldots\lambda_n}{(s+\lambda_0)(s+\lambda_1)\ldots(s+\lambda_n)}\cdot$
\subsection{\bf Resummation of series of fuzzy numbers via generalized Dirichlet series}
\begin{definition}
Given a sequence $0\leq\lambda_0<\lambda_1<\cdots\to\infty$, a series $\sum u_n$ of fuzzy numbers is said to be $(A,\lambda)$ summable to a fuzzy number $\mu$ if generalized Dirichlet series $\sum u_n e^{-\lambda_ns}$ converges for $s>0$ and
\begin{eqnarray*}
\lim_{s\to 0^+}\sum_{n=0}^{\infty}u_n e^{-\lambda_ns}=\mu.
\end{eqnarray*}
\end{definition}
We note that a convergent series of fuzzy numbers is $(A,\lambda)$ summable by Theorem \ref{regularity} but converse statement is not necessarily to hold which can be seen by series $\sum u_n$ of fuzzy numbers whose general term is defined by
$$u_n(t)=
\begin{cases}
(n+1)^2(t+(-1)^{n+1}), \quad &(-1)^{n}\leq t\leq (-1)^{n}+ (n+1)^{-2}\\
2-(n+1)^2\left(t+(-1)^{n+1}\right),\quad &(-1)^{n}+ (n+1)^{-2}\leq t\leq(-1)^{n}+ 2(n+1)^{-2}\\
0,& (otherwise).
\end{cases}$$
Series $\sum_{n=0}^{\infty}u_n$ of fuzzy numbers is $(A,\ln (n+1))$ summable to fuzzy number
$$\mu(t)=
\begin{cases}
\frac{6\left(t-1/2\right)}{\pi^2}, \quad &1/2\leq t\leq 1/2+\frac{\pi^2}{6}\\
2-\frac{6\left(t-(1/2)\right)}{\pi^2}, \quad &1/2+\frac{\pi^2}{6}\leq t\leq 1/2+\frac{\pi^2}{3}\\
0,& (otherwise),
\end{cases}$$
but it is not convergent.
\begin{theorem}\label{tauber1}
If series $\sum u_n$ of fuzzy numbers is $(A,\lambda)$ summable to fuzzy number $\mu$ and $\frac{\lambda_n}{\lambda_n-\lambda_{n-1}}D(u_n,\bar{0})=o(1)$, then $\sum u_n=\mu$.
\end{theorem}
\begin{proof}
Let series $\sum u_n$ of fuzzy numbers be $(A,\lambda)$ summable to $\mu$ and $\frac{\lambda_n}{\lambda_n-\lambda_{n-1}}D(u_n,\bar{0})=o(1)$ as $n\to\infty$. Then, since $\lim_{s\to 0^+}\sum_{k=0}^{\infty}u_k e^{-\lambda_ks}=\mu$ we have $\lim_{n\to \infty}\sum_{k=0}^{\infty}u_k e^{-\lambda_k/\lambda_n}=\mu$ in view of the fact that $1/\lambda_n\downarrow 0$ as $n\to\infty$. Hence, we get
\begin{eqnarray*}
D\left(\sum_{k=0}^n u_k, \mu\right) &\leq& D\left(\sum_{k=0}^n u_k, \sum_{k=0}^{\infty}u_ke^{-\lambda_k/\lambda_n}\right) +D\left(\sum_{k=0}^{\infty}u_ke^{-\lambda_k/\lambda_n},\mu\right)
\\&\leq&
\sum_{k=0}^n\left(1-e^{-\lambda_k/\lambda_n}\right)D(u_k,\bar{0})+\sum_{k=n+1}^{\infty}D(u_k,\bar{0})e^{-\lambda_k/\lambda_n}+D\left(\sum_{k=0}^{\infty}u_ke^{-\lambda_k/\lambda_n}, \mu\right)
\\&\leq&
\frac{1}{\lambda_n}\sum_{k=0}^n\lambda_k D(u_k,\bar{0}) +\sum_{k=n+1}^{\infty}\frac{\lambda_kD(u_k,\bar{0})}{\lambda_k-\lambda_{k-1}}\frac{\lambda_k-\lambda_{k-1}}{\lambda_k}e^{-\lambda_k/\lambda_n} +D\left(\sum_{k=0}^{\infty}u_ke^{-\lambda_k/\lambda_n}, \mu\right)
\\&=&
o(1)\left\{\frac{1}{\lambda_n}\sum_{k=0}^n(\lambda_k -\lambda_{k-1}) +\sum_{k=n+1}^{\infty}\frac{\lambda_k-\lambda_{k-1}}{\lambda_k}e^{-\lambda_k/\lambda_n}\right\}
\\&=&
o(1)\sum_{k=n+1}^{\infty}\frac{\lambda_k-\lambda_{k-1}}{\lambda_k}e^{-\lambda_k/\lambda_n}=\frac{o(1)}{\lambda_{n+1}}\sum_{k=n+1}^{\infty}e^{-\lambda_k/\lambda_n}\int_{\lambda_{k-1}}^{\lambda_k}dx
\\&=&
\frac{o(1)}{\lambda_{n+1}}\sum_{k=n+1}^{\infty}\int_{\lambda_{k-1}}^{\lambda_k}e^{-x/\lambda_n}dx=\frac{o(1)}{\lambda_{n+1}}\int_{\lambda_{n}}^{\infty}e^{-x/\lambda_n}dx=o(1)\left\{\frac{\lambda_{n}}{e\lambda_{n+1}}\right\}
\\&=&
o(1),
\end{eqnarray*}%
with agreements $n\geq1$ and $\lambda_{-1}=0$, which completes the proof.
\end{proof}
\noindent Replacing $o(1)$ with $O(1)$ in the proof of Theorem \ref{tauber1}, we can get following theorem.
\begin{theorem}
If series $\sum u_n$ of fuzzy numbers is $(A,\lambda)$ summable and $\frac{\lambda_n}{\lambda_n-\lambda_{n-1}}D(u_n,\bar{0})=O(1)$, then $(u_n)\in bs(F)$.
\end{theorem}

\subsection{\bf Resummation of series of fuzzy numbers via generalized factorial series}
\begin{definition}
Given a sequence $0<\lambda_0<\lambda_1<\cdots\to\infty$ where $\sum \frac{1}{\lambda_n}$ diverges, a series $\sum u_n$ of fuzzy numbers is said to be summable by generalized factorial series to a fuzzy number $\mu$ if generalized factorial series $\sum \frac{u_n\lambda_0\ldots\lambda_n}{(s+\lambda_0)\ldots(s+\lambda_n)}$ converges for $s>0$ and
\begin{eqnarray*}
\lim_{s\to 0^+}\sum_{n=0}^{\infty} \frac{u_n\lambda_0\ldots\lambda_n}{(s+\lambda_0)\ldots(s+\lambda_n)}=\mu.
\end{eqnarray*}
\end{definition}
We note that a convergent series of fuzzy numbers is summable via generalized factorial series by Theorem \ref{regularity} but converse statement is not necessarily to hold which can be seen by series $\sum u_n$ of fuzzy numbers where
$$u_n(t)=
\begin{cases}
(n+1)^4(t+(-1)^{n+1}(n+1)), &(-1)^{n}(n+1)\leq t\leq (-1)^{n}(n+1)+ (n+1)^{-4}\\
2-(n+1)^4\left(t+(-1)^{n+1}(n+1)\right), &(-1)^{n}n+ (n+1)^{-4}\leq t\leq(-1)^{n}n+ 2(n+1)^{-4}\\
0,& (otherwise).
\end{cases}$$%
Series $\sum_{n=0}^{\infty}u_n$ of fuzzy numbers is summable via ordinary factorial series(with choice $\lambda_n=n+1$) to fuzzy number
$$\mu(t)=
\begin{cases}
\frac{90\left(t-1/4\right)}{\pi^4}, \quad &1/4\leq t\leq 1/4+\frac{\pi^4}{90}\\
2-\frac{90\left(t-(1/4)\right)}{\pi^4}, \quad &1/4+\frac{\pi^4}{90}\leq t\leq 1/4+\frac{\pi^4}{45}\\
0,& (otherwise),
\end{cases}$$
but it is not convergent.

\begin{theorem}\label{tauber2}
If series $\sum u_n$ of fuzzy numbers is summable to fuzzy number $\mu$ by factorial series  and
$\lambda_n\left(\sum_{r=0}^{n}\frac{1}{\lambda_r}\right)D(u_n,\bar{0})=o(1)$, then $\sum u_n=\mu$.
\end{theorem}
\begin{proof}
Let series $\sum u_n$ of fuzzy numbers be summable by factorial series to $\mu$ and $\lambda_n\left(\sum_{r=0}^{n}\frac{1}{\lambda_r}\right)D(u_n,\bar{0})=o(1)$ as $n\to\infty$. Then, we get
\begin{eqnarray*}
D\left(\sum_{k=0}^n u_k, \mu\right)&\leq& D\left(\sum_{k=0}^n u_k, \sum_{k=0}^{\infty} \frac{u_k\lambda_0\ldots\lambda_k}{(s+\lambda_0)\ldots(s+\lambda_k)}\right) +D\left(\sum_{k=0}^{\infty} \frac{u_k\lambda_0\ldots\lambda_k}{(s+\lambda_0)\ldots(s+\lambda_k)},
\mu\right)
\\&\leq&
\sum_{k=0}^n\left(1-\frac{\lambda_0\ldots\lambda_k}{(s+\lambda_0)\ldots(s+\lambda_k)}\right)D(u_k,\bar{0})+\sum_{k=n+1}^{\infty}\frac{D(u_k,\bar{0})\lambda_0\ldots\lambda_k}{(s+\lambda_0)\ldots(s+\lambda_k)}
\\&\ \ &
+D\left(\sum_{k=0}^{\infty} \frac{u_k\lambda_0\ldots\lambda_k}{(s+\lambda_0)\ldots(s+\lambda_k)},\mu\right)
\\&\leq&
s\sum_{k=0}^nD(u_k,\bar{0})\sum_{r=0}^{k}\frac{1}{\lambda_r}+\sum_{k=n+1}^{\infty}D(u_k,\bar{0})e^{-\frac{s}{2}\sum_{r=0}^{k}\frac{1}{\lambda_r}}
+D\left(\sum_{k=0}^{\infty} \frac{u_k\lambda_0\ldots\lambda_k}{(s+\lambda_0)\ldots(s+\lambda_k)},\mu\right)
\end{eqnarray*}
in view of the inequalities (see \cite[p. 199]{fakultaten},\cite[p. 23]{main})
\begin{eqnarray*}
\left| \frac{\lambda_0\ldots\lambda_n}{(s+\lambda_0)\ldots(s+\lambda_n)}\right| < e^{-\frac{s}{2}\sum_{r=0}^{n}\frac{1}{\lambda_r}}, \qquad 1-\frac{\lambda_0\ldots\lambda_n}{(s+\lambda_0)\ldots(s+\lambda_n)}<|s|\sum_{r=0}^{n}\frac{1}{\lambda_r}\cdot
\end{eqnarray*}
Now letting $\gamma_k=\sum_{r=0}^{k}\frac{1}{\lambda_r}$, taking $s=\frac{1}{\gamma_n}$ and proceeding as in the proof of Theorem \ref{tauber1} we get
\begin{eqnarray*}
D\left(\sum_{k=0}^n u_k, \mu\right)&=&o(1)\left\{\frac{1}{\gamma_n}\sum_{k=0}^n(\gamma_k-\gamma_{k-1})+\sum_{k=n+1}^{\infty}\frac{\gamma_k-\gamma_{k-1}}{\gamma_k}e^{-\frac{\gamma_k}{2\gamma_n}}\right\}
=o(1)\left\{\frac{2\gamma_n}{\sqrt{e}\gamma_{n+1}}\right\}=o(1),
\end{eqnarray*}
which completes the proof.
\end{proof}
\noindent Replacing $o(1)$ with $O(1)$ in the proof of Theorem \ref{tauber2}, we may get following theorem.
\begin{theorem}
If series $\sum u_n$ of fuzzy numbers is summable by factorial series and $\lambda_n\left(\sum_{r=0}^{n}\frac{1}{\lambda_r}\right)D(u_n,\bar{0})=O(1)$, then $(u_n)\in bs(F)$.
\end{theorem}
\section{Resummation of level Fourier series of fuzzy valued functions}
\noindent Bede et al.\cite{fuzzyfourier} defined fuzzy Fourier sum of continuous fuzzy valued function $f:[-\pi,\pi]\to\mathbb{R}_{\mathcal{F}}$ by
\begin{eqnarray*}
\frac{a_0}{2}+\sum_{n=1}^{m}a_n\cos(nx)+b_n\sin(nx)
\end{eqnarray*}
where $a_0=\frac{1}{\pi}\int_{-\pi}^{\pi}f(x)dx$, $a_n=\frac{1}{\pi}\int_{-\pi}^{\pi}f(x)\cos(nx)dx$, $b_n=\frac{1}{\pi}\int_{-\pi}^{\pi}f(x)\sin(nx)dx$ and the integrals are in the fuzzy Riemann sense\cite{gv}. The $\alpha-$level set of a convergent fuzzy Fourier series of $f$ is of the form
\begin{eqnarray*}
\left[\frac{\{a_0\}^-_{\alpha}}{2}+\sum_{n=1}^{\infty}\left\{a_n\cos(nx)\right\}^-_{\alpha}+\left\{b_n\sin(nx)\right\}^-_{\alpha},
\frac{\{a_0\}^{+}_{\alpha}}{2}+\sum_{n=1}^{\infty}\left\{a_n\cos(nx)\right\}^{+}_{\alpha}+\left\{b_n\sin(nx)\right\}^{+}_{\alpha}\right].
\end{eqnarray*}%
By Theorem \ref{rep} and Definition \ref{series} , $\alpha-$level set of a convergent fuzzy Fourier series represents a fuzzy number. However, it is a bit challenging to express $\{a_n\cos(nx)\}^{\mp}_{\alpha}$ and $\{b_n\sin(nx)\}^{\mp}_{\alpha}$ in terms of $f^{-}_{\alpha}(x)$ and $f^{+}_{\alpha}(x)$ explicitly due to the effect of changing signs of $\sin(nx),\cos(nx)$ on interval $[-\pi,\pi]$(see $(ii)$ of \eqref{operations}). To tackle this problem, we now define a new type Fourier series for fuzzy valued functions which may be useful in calculations via $f^{-}_{\alpha}(x)$ and $f^{+}_{\alpha}(x)$.
\begin{definition}
Let $f$ be a $2\pi$-periodic continuous fuzzy valued function. Then, level Fourier series of $f$ is defined by the pair of series
\begin{eqnarray}\label{levelfourier}
\frac{a^{\mp}_0(\alpha)}{2}+\sum_{n=1}^{\infty}a_n^{\mp}(\alpha)\cos(nx)+b^{\mp}_n(\alpha)\sin(nx)
\end{eqnarray}%
where $a^{\mp}_0(\alpha)=\frac{1}{\pi}\int_{-\pi}^{\pi}f^{\mp}_{\alpha}(x)dx$, $a^{\mp}_n(\alpha)=\frac{1}{\pi}\int_{-\pi}^{\pi}f^{\mp}_{\alpha}(x)\cos(nx)dx$ and $b^{\mp}_n(\alpha)=\frac{1}{\pi}\int_{-\pi}^{\pi}f^{\mp}_{\alpha}(x)\sin(nx)dx$. In \eqref{levelfourier}, "-" and "+" signed series are called left and right level Fourier series of $f$, respectively.
\end{definition}
\noindent Left and right level Fourier series of a fuzzy valued function $f$ are expressed in terms of  $f^{-}_{\alpha}(x)$ and $f^{+}_{\alpha}(x)$, respectively. Hence, using level Fourier series is advantageous in calculations through the endpoints of $\alpha-$level set of $f$. However, it is disadvantageous in that the pair in \eqref{levelfourier} may not represent a fuzzy valued function(see Theorem \ref{rep}). Besides, as the common problem of convergence of Fourier series, level Fourier series may not converge to $f^{\mp}_{\alpha}(x)$ even in the pointwise sense. At this point we may use summation methods to achieve both fuzzification and convergence of level Fourier series. Let consider the summation method $(A,n)$. Since \eqref{levelfourier} is of the form
\begin{eqnarray*}
f^{\mp}_{\alpha}(x)&\sim&\frac{1}{2\pi}\int_{-\pi}^{\pi}f^{\mp}_{\alpha}(t)dt+\frac{1}{2\pi}\sum_{n=1}^{\infty}\int_{-\pi}^{\pi}f^{\mp}_{\alpha}(t)2\cos(n(x-t))dt
\\&=&
\frac{1}{2\pi}\int_{-\pi}^{\pi}f^{\mp}_{\alpha}(x-t)dt+\frac{1}{2\pi}\sum_{n=1}^{\infty}\int_{-\pi}^{\pi}f^{\mp}_{\alpha}(x-t)2\cos(nt)dt
\\&=&
\sum_{n=0}^{\infty}\dot{f}_n^{\mp}(\alpha,x)
\end{eqnarray*}
where
\begin{eqnarray*}
\dot{f}_n^{\mp}(\alpha,x)=
\begin{cases}
\frac{1}{2\pi}\int_{-\pi}^{\pi}f^{\mp}_{\alpha}(x-t)dt, \quad &n=0\\[3mm]
\frac{1}{\pi}\int_{-\pi}^{\pi}f^{\mp}_{\alpha}(x-t)\cos(nt)dt,\quad &n\neq0,
\end{cases}
\end{eqnarray*}
$(A,n)$ mean of level Fourier series of $f$ is
\begin{eqnarray*}
\sum \dot{f}^{\mp}_ne^{-ns}&=&\frac{1}{2\pi}\int_{-\pi}^{\pi}f^{^\mp}_{\alpha}(x-t)dt+\frac{1}{2\pi}\sum_{n=1}^{\infty}\left\{\int_{-\pi}^{\pi}f^{\mp}_{\alpha}(x-t)2\cos(nt)dt\right\}e^{-ns}\\
&=&\frac{1}{2\pi}\int_{-\pi}^{\pi}f^{^\mp}_{\alpha}(x-t)\left\{1+2\sum_{n=1}^{\infty}\cos(nt)e^{-ns}\right\}dt\\
&=&\frac{1-e^{-2s}}{2\pi}\int_{-\pi}^{\pi}\frac{f^{^\mp}_{\alpha}(x-t)dt}{1-2e^{-s}\cos t+e^{-2s}}\cdot
\end{eqnarray*}
The pair $\left\{\sum \dot{f}^{\mp}_ne^{-ns}\right\}(\alpha,x)$ satisfy the conditions of Theorem \ref{rep} for each $x\in\mathbb{R}, s>0$ and so represent the fuzzy valued function
\begin{eqnarray*}
\frac{1-e^{-2s}}{2\pi}\int_{-\pi}^{\pi}\frac{f(x-t)dt}{1-2e^{-s}\cos t+e^{-2s}}\cdot
\end{eqnarray*}
 Now we aim to show that
\begin{eqnarray*}
\lim_{s\to 0^+}\frac{1-e^{-2s}}{2\pi}\int_{-\pi}^{\pi}\frac{f(x-t)dt}{1-2e^{-s}\cos t+e^{-2s}}=f(x).
\end{eqnarray*}%
It will be convenient to utilize sequential criterion for limits in metric spaces and to use Theorem \ref{theorem1}. Let $(s_n)$ be a sequence in $\mathbb{R}^+$ that converges to $0$ and consider the sequence of fuzzy positive linear operators
\begin{eqnarray*}
L_n(f;x)=\frac{1-e^{-2s_n}}{2\pi}\int_{-\pi}^{\pi}\frac{f(x-t)dt}{1-2e^{-s_n}\cos t+e^{-2s_n}}\cdot
\end{eqnarray*}
Since
\begin{eqnarray*}
\widetilde{L}_n\left(f_{\alpha}^{\mp};x\right)=\frac{1-e^{-2s_n}}{2\pi}\int_{-\pi}^{\pi}\frac{f^{\mp}_{\alpha}(x-t)dt}{1-2e^{-s_n}\cos t+e^{-2s_n}}
\end{eqnarray*}
we have
\begin{eqnarray*}
\widetilde{L}_n(1;x)=1,\quad \widetilde{L}_n(\cos t;x)=e^{-s_n}\cos x,\quad \widetilde{L}_n(\sin t;x)=e^{-s_n}\sin x,
\end{eqnarray*}
and followingly we get
\begin{eqnarray*}
\left\Vert \widetilde{L}_n(f_0)-f_0\right\Vert=0, \quad \left\Vert\widetilde{L}_n(f_1)-f_1\right\Vert=1-e^{-s_n}\to 0, \quad \left\Vert\widetilde{L}_n(f_2)-f_2\right\Vert=1-e^{-s_n}\to 0.
\end{eqnarray*}
So by Theorem \ref{theorem1}, we have $L_n(f)\to f$. Since $(s_n)$ is arbitrary, we conclude
\begin{eqnarray}\label{operator}
\lim_{s\to 0^+}\frac{1-e^{-2s}}{2\pi}\int_{-\pi}^{\pi}\frac{f(x-t)dt}{1-2e^{-s}\cos t+e^{-2s}}=f(x).
\end{eqnarray}
Hence, fuzzification of level Fourier series \eqref{levelfourier} is recovered via $(A, n)$ means and convergence is achieved via \eqref{operator}.

On the other hand, by taking $r=e^{-s}$ above, Abel means of level Fourier series of $f$ represents the  fuzzy Abel-Poisson convolution operator\cite{yavuzbukres}
\begin{eqnarray*}
P_{r}(f;x)=\frac{1-{r}^2}{2\pi}\int_{-\pi}^{\pi}\frac{f(x-t)}{1-2{r}\cos t+{r}^2}{dt}
\end{eqnarray*}
and $\lim_{r\to1^-}P_{r}(f;x)=f(x)$ is achieved in view of \eqref{operator}. Fuzzy Abel-Poisson convolution operator satisfies conditions \eqref{poissonbound} and \eqref{poissonlimit} of the fuzzy Dirichlet problem(in the polar coordinates) on the unit disc
\begin{numcases}{}
\frac{\partial^2u}{\partial r^2}(r,x)+\frac{1}{r}\frac{\partial u}{\partial r}(r,x)+\frac{1}{r^2}\frac{\partial^2u}{\partial x^2}(r,x)=0, \quad x\in\mathbb{R},\  0<r<1 \label{Laplace}\\[2mm]
u(0,x)=\frac{1}{2\pi}\int_{-\pi}^{\pi}f(y)dy,\qquad\qquad\qquad \qquad \ \ \ x\in\mathbb{R}\label{poissonbound}\\[2mm]
\lim\limits_{r\to 1^-}u(r,x)=f(x),\quad f\in C_{2\pi}^{(\mathcal{F})}(\mathbb{R}).\label{poissonlimit}
\end{numcases}
It also satisfies fuzzy Laplace equation \eqref{Laplace} when the fuzzy derivatives are taken in the first form(see \cite[Theorem 5]{derivative}). However, we should emphasize that $u(r,x)=P_{r}(f;x)$ is a levelwise solution for \eqref{Laplace} since the first form fuzzy derivatives of $P_{r}(f;x)$ may not reveal a fuzzy valued function.
\section{Conclusion}
\noindent In current paper we have introduced a general summation method $(\phi)$ for series of fuzzy numbers via sequences of continuous functions and proved the regularity of $(\phi)$ method. We also obtained conditions which guarantee the convergence of a series of fuzzy numbers from its generalized Dirichlet and generalized factorial series. In $(A,\lambda)$ summability method if we choose $\lambda_n=n$ and use transformation $x=e^{-s}$, results for Abel summability of series of fuzzy numbers are obtained\cite[Theorem 12, Theorem14]{abel}. Similarly, if we choose $\phi_n(s)=\frac{1}{\Gamma(1+sn)}$ we obtain regularity result for Mittag-Leffler summation method for series of fuzzy numbers by Theorem \ref{regularity}. Besides, in $(A,\lambda)$ method  if we choose $\lambda_n=\ln n$ and $\lambda_n=n\ln n$($n\geq 1$) we get new results for summation of series of fuzzy numbers via ordinary Dirichlet series and via Lindel\"{o}f summation method, respectively.
\begin{theorem}
If series $\sum_{n=1}^{\infty} u_n$ of fuzzy numbers is summable to fuzzy number $\mu$ by ordinary Dirichlet series and $n\ln nD(u_n,\bar{0})=o(1)$, then $\sum u_n=\mu$.
\end{theorem}
\begin{theorem}
If series $\sum_{n=1}^{\infty} u_n$ of fuzzy numbers is summable by ordinary Dirichlet series and $n\ln nD(u_n,\bar{0})=O(1)$, then $(u_n)\in bs(F)$.
\end{theorem}
\begin{theorem}
If series $\sum_{n=1}^{\infty} u_n$ of fuzzy numbers is Lindel\"{o}f summable to fuzzy number $\mu$ and $nD(u_n,\bar{0})=o(1)$, then $\sum u_n=\mu$.
\end{theorem}
\begin{theorem}
If series $\sum_{n=1}^{\infty} u_n$ of fuzzy numbers is Lindel\"{o}f summable and $nD(u_n,\bar{0})=O(1)$, then $(u_n)\in bs(F)$.
\end{theorem}

\end{document}